\documentclass[10pt,onecolumn]{IEEEtran}

\usepackage{amsbsy}
\usepackage{amsmath}
\usepackage{amsthm}
\usepackage{amssymb}
\usepackage{latexsym}
\usepackage{ifthen}
\usepackage{subfigure}
\usepackage{turnstile}
\usepackage{mathtools}
\usepackage{booktabs}
\usepackage{balance}
\usepackage{paralist}
\usepackage{cite}
\usepackage{url}
\usepackage{color}
\usepackage{cite}

\newtheorem{theorem}{Theorem}
\newtheorem{remark}{Remark}

\newtheorem{assumption}{Assumption}

\newtheorem{lemma}{Lemma}

\newcommand{\vecc}{\boldsymbol{\operatorname{Vec}}}

  \def\vw{{\bf w}} \def\vx{{\bf x}}

\def \Vap{\varepsilon}

\def \obx{\overline{\mathbf{x}}}

\def \omw{\overline{\mathbf{w}}}
\def \brx{\breve{\mathbf{x}}}

\def \bU{\overline{U}}
\def \oR{\overline{R}}
\def \br{\overline{r}}

\def \bzeta{\boldsymbol{\zeta}}
\def \obzeta{\overline{\boldsymbol{\zeta}}}
\def \bxi{\boldsymbol{\xi}}
\def \obxi{\overline{\boldsymbol{\xi}}}

\def \e{\varepsilon}
\def \R{\mathbb{R}}
\def \dx{\,dx}
\def \E{\mathbb{E}}

\def \calF{\mathcal{F}}
\def \calG{\mathcal{G}}

\mathtoolsset{showonlyrefs}

\title{Annealing for Distributed Global Optimization}

\author{Brian Swenson, Soummya Kar, H. Vincent Poor, and Jos\'{e} M. F. Moura\thanks{This work was partially supported by the Air Force Office of Scientific Research under MURI Grant FA9550-18-1-0502 and was partially supported by the National Science Foundation under Award Number CCF. 1513936 \newline
Brian Swenson and H. Vincent Poor are with the Department of Electrical Engineering, Princeton University, Princeton, NJ 08540 {\texttt{bswenson@princeton.edu, poor@princeton.edu}}.\newline
Soummya Kar and Jos\'{e} M. F. Moura are with the Department of Electrical and Computer Engineering, Carnegie Mellon University, Pittsburgh, PA 15213 {\texttt{soummyak@andrew.cmu.edu, moura@andrew.cmu.edu}}}}

\begin{document}

\maketitle

\renewcommand{\thefootnote}{\fnsymbol{footnote}}

\renewcommand{\thefootnote}{\arabic{footnote}}

\begin{abstract}
The paper proves convergence to global optima for a class of distributed algorithms for nonconvex optimization in network-based multi-agent settings. Agents are permitted to communicate over a time-varying undirected graph. Each agent is assumed to possess a local objective function (assumed to be smooth, but possibly nonconvex). The paper considers algorithms for optimizing the sum function.
A distributed algorithm of the consensus+innovations type is proposed which relies on first-order information at the agent level. Under appropriate conditions on network connectivity and the cost objective, convergence to the set of global optima is achieved by an annealing-type approach, with decaying Gaussian noise independently added into each agent's update step. It is shown that the proposed algorithm converges in probability to the set of global minima of the sum function.
\end{abstract}


\thispagestyle{plain}
\markboth{}{}

\section{Introduction}
\label{introduction}
In this paper we consider a class of algorithms for nonconvex optimization in distributed multi-agent systems and prove convergence to the set of global optima.
Recent years have seen a surge in research interest in nonconvex optimization, motivated, to a large degree, by emerging applications in machine learning and artificial intelligence.
The majority of research in this area has focused on centralized computing frameworks in which memory and processing resources are either shared or coordinated by a central mechanism \cite{bottou2010large,dean2008mapreduce,zinkevich2010parallelized,ge2015escaping,kingma2014adam,jin2017escape,dauphin2014identifying,gelfand1991recursive,murray2018revisiting}.

With the advent of the internet of things (IoT) and low-latency 5G communication networks, there is a growing trend towards storing and processing data at the ``edge'' of the network (e.g., directly on IoT devices) rather than processing data in the cloud.
This necessitates algorithms that are able to operate robustly in adhoc networked environments without centralized coordination. Beyond applications in IoT, distributed algorithms for non-convex optimization also play an important role in other domains, including power systems \cite{guo2017case}, sensor networks \cite{biswas2006semidefinite}, unmanned aerial vehicles \cite{jun2003path}, and wireless communications \cite{bianchi2013convergence}.

This paper considers the following distributed computation framework: A group of $N$ agents (or nodes) communicates over a (possibly random, possibly sparse) undirected communication graph $G$. Each agent has a local objective function $U_n:\R^d\to\R$. We are interested in distributed algorithms that optimize the sum function
\begin{equation} \label{eq:U-def}
U(x) = \frac{1}{N}\sum_{n=1}^N U_n(x)
\end{equation}
using only local neighborhood information exchange between agents and without any centralized coordination.

As an example, in the context of distributed risk minimization or probably approximately correct (PAC) learning, e.g., \cite{lee2018distributed}, the $U_n(\cdot)$'s may correspond to (expected) risk
$$
U_{n}(\theta) = \E_{D_{n}}[l_{n}(\theta,z_{n})],
$$
where $l_{n}(\cdot,\cdot)$ is the local loss function at agent $n$ and $D_{n}$ is the local data distribution. The agents are interested in learning a common ``hypothesis,'' parameterized by $\theta$, using their collective data.

Distributed optimization algorithms have been studied extensively when the objective functions are convex \cite{rabbat2004distributed,nedic2009distributed,jakovetic2014fast,nedic2015distributed,
chen2012diffusion,gharesifard2014distributed}.
%for \emph{convex} optimization have been studied extensively over the last decade [cite several papers].
Not so when the objective is \emph{non-convex}. The majority of current work in this area focuses on demonstrating convergence of distributed algorithms to critical points of $U$ (not necessarily to minima, local or global).

This motivates us to consider a class of distributed algorithms for computing the global optima of \eqref{eq:U-def}. Our algorithms take the form:
\begin{align} \label{eq:update}
%\label{eq:update}
\mathbf{x}_{n}(t+1) ~=~ & \mathbf{x}_{n}(t)-\beta_{t}\sum_{l\in\Omega_{n}(t)}\left(\mathbf{x}_{n}(t)-\mathbf{x}_{l}(t)\right)\\
& -\alpha_{t}\left(\nabla U_{n}(x_{n}(t))+\bzeta_{n}(t)\right)+\gamma_{t}\mathbf{w}_{n}(t),
\end{align}
$n=1,\ldots,N$, where $\vx_n(t)\in\R^d$ is the state of agent $n$ at iteration $t\geq 0$, $\Omega_n(t)$ denotes the set of agents neighboring agent $n$ at time $t$ (per the communication graph), $\{\alpha_t\}$ and $\{\beta_t\}$ are sequences of decaying weight parameters, $\{\gamma_t\}$ is a sequence of decaying annealing weights, $\zeta_n$(t) is a $d$-dimensional random variable (representing gradient noise), and $\vw_n(t)$ is a $d$-dimensional Gaussian noise (introduced for annealing).
The algorithm is distributed since in \eqref{eq:update} each agent only knows its local function $U_n(\cdot)$ and accesses information on the state of neighboring agents.

Algorithm \eqref{eq:update} may be viewed as a distributed consensus + innovations algorithm \cite{kar2012distributed}. The algorithm consists of the consensus term, $-\beta_{t}\sum_{l\in\Omega_{n}(t)}\left(\mathbf{x}_{n}(t)-\mathbf{x}_{l}(t)\right)$, that encourages agreement among agents, and the innovation term, $-\alpha_{t}\left(\nabla U_{n}(x_{n}(t))+\bzeta_{n}(t)\right)$, that encourages each agent to follow the gradient descent direction of their local objective function (with $\bzeta_t$ being zero-mean gradient noise). Finally, the term $\gamma_{t}\mathbf{w}_{n}(t)$ is an annealing term that injects decaying Gaussian noise into the dynamics to destabilize local minima and saddle points.
By appropriately controlling the decay rates of the parameter sequences, one can balance the various objectives of reaching consensus among agents, reaching a critical point of \eqref{eq:U-def}, and destabilizing local minima and saddle points (see Assumption~\ref{ass:weights}).

Our main contribution is the following: We show that, under appropriate assumptions (outlined below), the distributed algorithm \eqref{eq:update} converges in probability to the set of global minima of \eqref{eq:U-def}. More precisely, it will be shown that (i) agents reach consensus, almost surely (a.s.), i.e., $\lim_{t\to\infty}\|\vx_n(t) - \vx_{\ell}(t)\| = 0$ for each $n,\ell=1,\ldots,N$, a.s., and (ii) for each agent $n$, $\vx_n(t)$ converges in probability to the set of global minima of $U(\cdot)$.
%That is, each $\vx_n(t)$ converges weakly to a random variable that places mass 1 on the set of global minimizers of \eqref{eq:U-def}.
A precise statement of the main result is given in Theorem~\ref{th:main_result} at the end of Section~\ref{sec:main_res}.

Theorem~\ref{th:main_result} is proved under Assumptions \ref{ass:Lip}--\ref{ass:GM_6}.
Assumptions \ref{ass:Lip}--\ref{ass:similarity} and \ref{ass:GM_1}--\ref{ass:GM_6} concern the agents' objective functions, Assumption~\ref{ass:conn} concerns the time-varying communication graph, Assumptions \ref{ass:grad_noise}--\ref{ass:gauss} concern the gradient annealing noise, and Assumption~\ref{ass:weights} concerns the weight parameter sequences.

\bigskip
\noindent \textbf{Related Work}.
Work on distributed optimization with convex objectives has been studied extensively; for an overview of the expansive literature in this field we refer readers to \cite{rabbat2004distributed,nedic2009distributed,jakovetic2014fast,nedic2015distributed,
chen2012diffusion,gharesifard2014distributed} and references therein.

The topic of distributed algorithms for non-convex optimization is %less developed and
a subject of more recent research focus. We briefly summarize related contributions here. Reference \cite{bianchi2013convergence} considers an algorithm for nonconvex optimization (possibly constrained) over an undirected communication graph and shows convergence to KKT points. Relevant applications to wireless adhoc networks are discussed. Reference \cite{zhu2013approximate} considers a distributed primal dual algorithm for nonconvex optimization. The primal dual algorithm solves an approximation to the original nonconvex problem. Reference \cite{magnusson2016convergence} analyzes the alternating direction penalty method and method of multipliers in nonconvex problems and demonstrates convergence to primal feasible points under mild assumptions. Reference \cite{tatarenko2017non} considers a push-sum algorithm for distributed nonconvex optimization on time-varying directed graphs and demonstrates convergence to first-order stationary points. \cite{di2016next} considers a distributed algorithm for nonconvex optimization with smooth objective and possibly non-smooth regularizer and demonstrates convergence to stationary solutions. Our work differs from these primarily in that we study distributed algorithms for \emph{global} optimization of a nonconvex function.

The key feature of this approach is the incorporation of decaying Gaussian noise that allows the algorithm to escape local minima. Such techniques were explored in \cite{gelfand1991recursive} and later studied and successfully applied in various centralized settings; e.g., \cite{kushner2003stochastic,maryak2001global,raginsky2017non,andradottir1998review,kushner1987asymptotic} and references therein. On the other hand, consensus + innovations techniques, such as those used in \cite{kar2012distributed,kar2013distributed}, are used in distributed settings.
In this paper we prove global optimal convergence for consensus + innovations techniques, with an appropriate annealing schedule, in nonconvex optimization. This results in a distributed equivalent of the centralized result in \cite{gelfand1991recursive}.

As noted in \cite{kar2011convergence}, the analysis techniques developed to study consensus + innovations algorithms contributes to the general theory of mixed-time-scale stochastic approximation (SA) algorithms, e.g., \cite{borkar1997stochastic}. In such algorithms, the right-hand side of the stochastic approximation difference equation contains two potentials decaying at different rates. The work \cite{gelfand1991recursive} studies mixed-time scale SA algorithms in the context of simulated annealing. In \cite{gelfand1991recursive}, the term that serves a role analogous to our innovations potential is assumed to converge asymptotically to a Martingale difference process. A key element of our analysis here is to characterize the rate at which the innovation potential converges to a Martingale difference sequence in order to apply the results of \cite{gelfand1991recursive}.

% The analysis approach that we develop is of independent interest and contributes to the theory of mixed-time-scale stochastic approximation. The work \cite{gelfand1991recursive} develops methods to analyze mixed-time scale algorithms in the context of simulated annealing. In \cite{gelfand1991recursive}, the role of our innovation potential is played by a term which asymptotically converges to a Martingale difference process. A key step in our analysis is to derive approximation results in order to characterize the rate at which the innovation process converges to a Martingale difference processes so that we may apply the results of \cite{gelfand1991recursive}.

\bigskip
\noindent \textbf{Organization}.
The remainder of the paper is organized as follows. Section~\ref{notgraph} introduces relevant notation. Section~\ref{sec:alg} formally introduces our distributed algorithm. Section~\ref{sec:tech_lemmas} presents the assumptions used in our main result and some intermediate results, and it reviews a classical result in global optimization (Theorem~\ref{th:GM}) that will be used in the proof of our main result. Section~\ref{sec:main_res} proves the main result (Theorem~\ref{th:main_result}).

\subsection{Notation}
\label{notgraph} The set of reals is denoted by $\mathbb{R}$, whereas $\mathbb{R}_{+}$ denotes the non-negative reals. For $a,b\in\mathbb{R}$, we will use the notations $a\vee b$ and $a\wedge b$ to denote the maximum and minimum of $a$ and $b$ respectively. We denote the $k$-dimensional Euclidean space by
$\mathbb{R}^{k}$. The set of $k\times k$ real matrices is denoted by $\mathbb{R}^{k\times k}$.
%\textcolor{red}{Needed? The corresponding subspace of symmetric matrices is denoted by $\mathbb{S}^{k}$. The cone of positive semidefinite matrices is denoted by $\mathbb{S}_{+}^{k}$, whereas $\mathbb{S}_{++}^{k}$ denotes the subset of positive definite matrices.}
The $k\times k$ identity matrix is
denoted by $I_{k}$, while $\mathbf{1}_{k}$ and $\mathbf{0}_{k}$ denote
respectively the column vector of ones and zeros in
$\mathbb{R}^{k}$. Often the symbol $0$ is used to denote the $k\times p$ zero matrix, the dimensions being clear from the context. The operator
$\left\|\cdot\right\|$ applied to a vector denotes the standard
Euclidean $\mathcal{L}_{2}$ norm, while applied to matrices it denotes the induced $\mathcal{L}_{2}$ norm, which is equivalent to the matrix spectral radius for symmetric matrices. The notation $A\otimes B$ is used for the Kronecker product of two matrices $A$ and $B$. We say that a function $f$ is of class $C^k$, $k\geq 1$, if $f$ is $k$-times continuously differentiable.

Given a set of elements in $z_1,\ldots,z_N$ belonging to some Euclidean space, we let $\vecc(\{z_i\}_{i=1}^N)$ denote the vector stacking these elements. To simplify notation, we sometimes suppress the interior brackets when the meaning is clear.

We assume there exists a rich enough probability space to carry out the constructions of the random objects defined in the paper. Unless stated otherwise, all (in)equalities involving random objects are to be interpreted almost surely (a.s.). We denote by $\mathbb{P}$ and $\mathbb{E}$ probability and expectation respectively. Given a measure $\pi$ on $\mathbb{R}^{k}$ and a (measurable) function $f:\mathbb{R}^{k}\mapsto\mathbb{R}$, we let
\begin{equation}\label{eq:pi-f-def}
\pi(f) ~\dot= \int f d\pi,
\end{equation}
whenever the integral exists.
For a stochastic process $\{Z_{t}\}$ and a function $f$, we let
\begin{equation} \label{eq:def-cond-E}
\E_{t_0,z_0}[f(Z_t)] ~ \dot= ~ \E[f(Z_t)\vert Z_{t_0} = z_0].
\end{equation}

\textbf{Spectral graph theory}: The inter-agent communication topology may be described by an \emph{undirected} graph $G=(V,E)$, with $V=\left[1\cdots N\right]$ and~$E$  the set of agents (nodes) and communication links (edges), respectively. The unordered pair $(n,l)\in E$ if there exists an edge between nodes~$n$ and~$l$. We consider simple graphs, i.e., graphs devoid of self-loops and multiple edges. A graph is connected if there exists a path\footnote{A path between nodes $n$ and $l$ of length $m$ is a sequence
$(n=i_{0},i_{1},\cdots,i_{m}=l)$ of vertices, such that $(i_{k},i_{k+1})\in E\:\forall~0\leq k\leq m-1$.}, between each pair of nodes. The neighborhood of node~$n$ is
\begin{equation}
\label{def:omega} \Omega_{n}=\left\{l\in V\,|\,(n,l)\in
E\right\}. %n\in \left[1\cdots N\right]
\end{equation}
Node~$n$ has degree $d_{n}=|\Omega_{n}|$ (the number of edges with~$n$ as one end point). The structure of the graph can be described by the symmetric $N\times N$ adjacency matrix, $A=\left[A_{nl}\right]$, $A_{nl}=1$, if $(n,l)\in E$, $A_{nl}=0$, otherwise. Let the degree matrix  be the diagonal matrix $D=\mbox{diag}\left(d_{1}\cdots d_{N}\right)$. The positive semidefinite matrix $L=D-A$ is the graph Laplacian matrix. The eigenvalues of $L$ can be ordered as $0=\lambda_{1}(L)\leq\lambda_{2}(L)\leq\cdots\leq\lambda_{N}(L)$, the eigenvector corresponding to $\lambda_{1}(L)$ being $(1/\sqrt{N})\mathbf{1}_{N}$. The multiplicity of the zero eigenvalue equals the number of connected components of the network; for a connected graph, $\lambda_{2}(L)>0$. This second eigenvalue is the algebraic connectivity or the Fiedler value of
the network; see \cite{chung1997spectral} for detailed treatment of graphs and their spectral theory.

\section{Algorithm}
\label{sec:alg}

Consider $N$ agents connected over a time-varying graph, with $L_{t}$ denoting the graph Laplacian at time $t$. Let $U_n:\R^d\to \R$, $n=1,\ldots,N$ denote the objective function of agent $n$. Let $U:\R^d\to \R$ be as defined in \eqref{eq:U-def}.
%$$
%U(x) = \frac{1}{N}\sum_{n=1}^N U_n(x).
%$$

The agents update their states in a distributed fashion according to \eqref{eq:update}
%\begin{align}
%\label{eq:update}
%\mathbf{x}_{n}(t+1)=\mathbf{x}_{n}(t)-\beta_{t}\sum_{l\in\Omega_{n}(t)}\left(\mathbf{x}_{n}(t)-\mathbf{x}_{l}(t)\right)-\alpha_{t}\left(\nabla U_{n}(x_{n}(t))+\bzeta_{n}(t)\right)+\gamma_{t}\mathbf{w}_{n}(t),
%\end{align}
for all $t\geq 0$ with deterministic initial conditions $\mathbf{x}_{n}(0)\in\mathbb{R}^{d}$, $n=1,\cdots,N$. In \eqref{eq:update}, $\bzeta_{n}(t)$ denotes gradient noise and $\mathbf{w}_{n}(t)$ denotes a standard normal vector (introduced for \emph{annealing}). In vector form, the update in~\eqref{eq:update} may be written as:
\begin{align}
\label{eq:vec_update}
\mathbf{x}_{t+1}= ~\mathbf{x}_{t} & -\beta_{t}\left(L_{t}\otimes I_{d}\right)\mathbf{x}_{t}\\
& -\alpha_{t}\left(\nabla\bU(\mathbf{x}_{t}) + \bzeta_{t}\right)+\gamma_{t}\mathbf{w}_{t},
\end{align}
where $\mathbf{x}_{t}=\vecc(\mathbf{x}_{n}(t))$, $\bU(\mathbf{x}_{t})=\vecc(U_{n}(\mathbf{x}_{n}(t)))$, $\bzeta_{t}=\vecc(\bzeta_{n}(t))$, $\mathbf{w}_{t}=\vecc(\mathbf{w}_{n}(t))$, and $L_{t}$ denotes the (stochastic) undirected graph Laplacian.

\begin{remark}
In empirical risk minimization, agents optimize an empirical risk function using collected data, rather than optimizing the expected risk. In such problems, it is common to use stochastic gradient descent (SGD) techniques that mitigate computational burden by handling the data in batches. We note that our framework readily handles such SGD techniques as the $\bzeta_t$ term can model independent gradient noise.
\end{remark}

\section{Intermediate Results}
\label{sec:tech_lemmas}
This section presents some intermediate results. In Section~\ref{sec:tech-lemmas}, we begin by presenting several technical lemmas. Subsequently, in Section~\ref{sec:consensus} we will use these technical lemmas to prove that the algorithm \eqref{eq:vec_update} obtains asymptotic consensus (see Lemma~\ref{lm:mean_conv_as}). Finally, in Section~\ref{sec:GM-result}, we briefly review classical results in global optimization that will be used in the proof of our main result.

\subsection{Technical Results} \label{sec:tech-lemmas}
We begin by making the following assumptions.
\begin{assumption}
\label{ass:Lip} The functions $U_{n}(\cdot)$ are $C^{2}$ with Lipschitz continuous gradients, i.e., there exists $L>0$ such that
\begin{align}
\left\|\nabla U_{n}(\mathbf{x})-\nabla U_{n}(\acute{\mathbf{x}})\right\|\leq L\left\|\mathbf{x}-\acute{\mathbf{x}}\right\|
\end{align}
for all $n$.
\end{assumption}
\begin{assumption}
\label{ass:similarity} The functions $U_{n}(\cdot)$ satisfy the following bounded gradient-dissimilarity condition:
\begin{align}
\sup_{\mathbf{x}\in\mathbb{R}^{d}}\left\|\nabla U_{n}(\mathbf{x})-\nabla U(\mathbf{x})\right\|<\infty,~~~\forall n.
\end{align}
\end{assumption}
%\begin{remark}
%\label{rem:similarity} \textcolor{red}{Note that Assumption~\ref{ass:similarity} holds trivially if the functions have bounded gradients. Unfortunately, though the global convergence results that we aim are typically shown~\cite{gelfand1991recursive} for functions with unbounded gradients. I think this assumption can still be relaxed by showing a tightness result (using more refined arguments as in Theorem 4.1 in~\cite{kar2014asymptotically}) and effectively restricting attention to a bounded region.}
%\end{remark}

Denote by $\{\mathcal{H}_{t}\}$ the natural filtration corresponding to the update process~\eqref{eq:update}, i.e., for all $t$, the $\sigma$-algebra $\mathcal{H}_t$ is given by
\begin{align}
\label{eq:filt}
\mathcal{H}_{t}=\sigma\left(\mathbf{x}_{0},L_{0},\cdots,L_{t-1},\bzeta_{0},\cdots,\bzeta_{t-1},\mathbf{w}_{0},\cdots,\mathbf{w}_{t-1}\right).
\end{align}

%We make the following assumptions:
\begin{assumption}
\label{ass:conn} The $\{\mathcal{H}_{t+1}\}$-adapted sequence of undirected graph Laplacians $\{L_{t}\}$ are independent and identically distributed (i.i.d.), with $L_{t}$ being independent of $\mathcal{H}_{t}$ for each $t$, and are connected in the mean, i.e., $\lambda_{2}(\bar{L})>0$ where $\bar{L}=\mathbb{E}[L_{t}]$.
\end{assumption}
\begin{assumption}
\label{ass:grad_noise} The sequence $\{\bzeta_{t}\}$ is $\{\mathcal{H}_{t+1}\}$-adapted and there exists a constant $C_{1}>0$ such that
\begin{align}
\label{ass:grad_noise1}
\mathbb{E}[\bzeta_{t}~|~\mathcal{H}_{t}]=0~~\mbox{and}~~\mathbb{E}[\|\bzeta_{t}\|^{2}~|~\mathcal{H}_{t}]<C_{1}
\end{align}
for all $t\geq 0$.
\end{assumption}
\begin{assumption}
\label{ass:gauss}
For each $n$, the sequence $\{\mathbf{w}_{n}(t)\}$ is a sequence of i.i.d. $d$-dimensional standard Gaussian vectors with covariance $I_d$ and with $\mathbf{w}_{n}(t)$  being independent of $\mathcal{H}_{t}$ for all $t$. Further, the sequences $\{\mathbf{w}_{n}(t)\}$ and $\{\mathbf{w}_{l}(t)\}$ are mutually independent for each pair $(n,l)$ with $n\neq l$.
\end{assumption}
\begin{assumption}
\label{ass:weights} The sequences $\{\alpha_{t}\}$, $\{\beta_{t}\}$, and $\{\gamma_{t}\}$ satisfy
\begin{align}
\label{eq:weights}
\alpha_{t}=\frac{c_{\alpha}}{t},~~\beta_{t}=\frac{c_{\beta}}{t^{\tau_{\beta}}},~~\gamma_{t}=\frac{c_{\gamma}}{t^{1/2}\sqrt{\log\log t}},~~~\mbox{for $t$ large},
\end{align}
where $c_{\alpha},c_{\beta},c_{\gamma}>0$ and $\tau_{\beta}\in (0,1/2)$.
\end{assumption}

The following lemma characterizes the decay rate of scaled gradient noise.
\begin{lemma}
\label{lm:zeta_decay}
Let Assumption~\ref{ass:grad_noise} hold. Then, for every $\delta>0$, we have that $(t+1)^{-1/2-\delta}\|\bzeta_{t}\|\rightarrow 0$ a.s. as $t\rightarrow\infty$.
\end{lemma}
\begin{proof}
Fix $\Vap>0$ and note that, by Assumption~\ref{ass:grad_noise},
\begin{align}
\mathbb{P}\left((t+1)^{-1/2-\delta}\|\bzeta_{t}\|>\Vap\right) & \leq \frac{1}{\Vap^{2}(t+1)^{1+2\delta}}\mathbb{E}\left[\|\bzeta\|^{2}\right]\\
&\leq \frac{C_{1}}{\Vap^{2}(t+1)^{1+2\delta}}.
\label{eq:zeta_decay1}
\end{align}
Since $\delta>0$, the term on the R.H.S. of~\eqref{eq:zeta_decay1} is summable, and by the Borel-Cantelli lemma we may conclude that
\begin{align}
\label{eq:zeta_decay2}
\mathbb{P}\left((t+1)^{-1/2-\delta}\|\bzeta_{t}\|>\Vap~\mbox{i.o.}\right)=0,
\end{align}
where i.o. means infinitely often.
Since $\Vap>0$ is arbitrary, the desired assertion follows.
\end{proof}

The following two technical results from~\cite{kar2013distributed}  will be useful (see also \cite{kar2012distributed}).

\begin{lemma}[Lemma 4.3 in~\cite{kar2013distributed}]
\label{lm:mean-conv} Let $\{\mathbf{z}_{t}\}$ be an $\mathbb{R}_{+}$ valued $\{\mathcal{H}_{t}\}$ adapted process that satisfies
\begin{equation}
\label{lm:mean-conv1}
\mathbf{z}_{t+1}\leq \left(1-r_{1}(t)\right)\mathbf{z}_{t}+r_{2}(t)V_{t}\left(1+J_{t}\right).
\end{equation}
In the above, $\{r_{1}(t)\}$ is an $\{\mathcal{H}_{t+1}\}$ adapted process, such that for all $t$, $r_{1}(t)$ satisfies $0\leq r_{1}(t)\leq 1$ and
\begin{equation}
\label{lm:JSTSP2}
\frac{a_{1}}{(t+1)^{\delta_{1}}}\leq\mathbb{E}\left[r_{1}(t)~|~\mathcal{H}_{t}\right]\leq 1
\end{equation}
with $a_{1}>0$ and $0\leq \delta_{1}< 1$. The sequence $\{r_{2}(t)\}$ is deterministic, $\mathbb{R}_{+}$ valued and satisfies $r_{2}(t)\leq a_{2}/(t+1)^{\delta_{2}}$ with $a_{2}>0$ and $\delta_{2}>0$.
Further, let $\{V_{t}\}$ and $\{J_{t}\}$ be $\mathbb{R}_{+}$ valued $\{\mathcal{H}_{t+1}\}$ adapted processes with $\sup_{t\geq 0}\|V_{t}\|<\infty$ a.s. The process $\{J_{t}\}$ is i.i.d.~with $J_{t}$ independent of $\mathcal{H}_{t}$ for each $t$ and satisfies the moment condition $\mathbb{E}\left[\left\|J_{t}\right\|^{2+\varepsilon_{1}}\right]<\kappa<\infty$ for some $\varepsilon_{1}>0$ and a constant $\kappa>0$. Then, for every $\delta_{0}$ such that
\begin{equation}
\label{lm:mean-conv5}
0\leq\delta_{0}<\delta_{2}-\delta_{1}-\frac{1}{2+\varepsilon_{1}},
\end{equation}
we have $(t+1)^{\delta_{0}}\mathbf{z}_{t}\rightarrow 0$ a.s. as $t\rightarrow\infty$.
\end{lemma}
\begin{remark}
\label{rem:lm:mean-conv} Note that Lemma 4.3 in~\cite{kar2013distributed} assumes $V_{t}$ above is $\{\mathcal{H}_{t}\}$ adapted, but the proof in~\cite{kar2013distributed} uses a pathwise analysis and can be readily adjusted to the case of $\{\mathcal{H}_{t+1}\}$ adapted $V_{t}$.
\end{remark}

In $\mathbb{R}^{Nd}$, denote by $\mathcal{C}$ the consensus subspace,
\begin{align}
\label{eq:cons_subs}
\mathcal{C} \dot=\left\{\mathbf{z}\in\mathbb{R}^{Nd}~:~\mathbf{z}=\mathbf{1}_{N}\otimes\mathbf{a}~\mbox{for some $\mathbf{a}\in\mathbb{R}^{d}$}\right\},
\end{align}
and denote by $\mathcal{C}^{\perp}$ its orthogonal subspace in $\mathbb{R}^{Nd}$.
\begin{lemma}[Lemma 4.4 in~\cite{kar2013distributed}] \label{lm:conn} Let $\{\mathbf{z}_{t}\}$ be an $\mathbb{R}^{Nd}$ valued $\{\mathcal{H}_{t}\}$ adapted process such that $\mathbf{z}_{t}\in\mathcal{C}^{\perp}$ for all $t$. Also, let $\{L_{t}\}$ be an i.i.d.~sequence of Laplacian matrices as in Assumption~\ref{ass:conn} that satisfies
\begin{equation}
\label{Lap_cond}
=\lambda_{2}\left(\mathbb{E}[L_{t}]\right)=\lambda_{2}(\overline{L})>0,
\end{equation}
with $L_{t}$ being $\mathcal{H}_{t+1}$ adapted and independent of $\mathcal{H}_{t}$ for all $t$. Then there exists a measurable $\{\mathcal{H}_{t+1}\}$ adapted $\mathbb{R}_{+}$ valued process $\{r_{t}\}$ (depending on $\{\mathbf{z}_{t}\}$ and $\{L_{t}\}$) and a constant $c_{r}>0$, such that $0\leq r_{t}\leq 1$ a.s. and
\begin{equation}
\label{lm:conn20}
\left\|\left(I_{Nd}-\beta_{t}L_{t}\otimes I_{d}\right)\mathbf{z}_{t}\right\|\leq\left(1-r_{t}\right)\left\|\mathbf{z}_{t}\right\|
\end{equation}
with
\begin{equation}
\label{lm:conn2}
\mathbb{E}\left[r_{t}~|~\mathcal{H}_{t}\right]\geq\frac{c_{r}}{(t+1)^{\tau_{\beta}}}~~\mbox{a.s.}
\end{equation}
for all $t$ large enough, where the weight sequence $\{\beta_{t}\}$ and $\tau_{\beta}$ are defined in~Assumption~\ref{ass:weights}.
\end{lemma}
See \cite{kar2013distributed} for a detailed discussion of the necessity of the various technicalities involved in the statement of Lemma~\ref{lm:conn}.

\subsection{Consensus} \label{sec:consensus}
The following lemma shows that, a.s., the algorithm \eqref{eq:vec_update} obtains consensus asymptotically.
\begin{lemma}[Convergence to Consensus Subspace]
\label{lm:mean_conv_as} Let Assumptions~\ref{ass:Lip}-\ref{ass:weights} hold. Let $\{x_t\}$ satisfy \eqref{eq:vec_update} with arbitrary initial condition. Then, for every $\tau\in [0,1/2-\tau_{\beta})$, we have that
\begin{align}
\label{lm:mean_conv_as1}
\mathbb{P}\left(\lim_{t\rightarrow\infty}(t+1)^{\tau}\left\|\mathbf{x}_{n}(t)-\obx_{t}\right\|=0\right)=1,~~~\forall n,
\end{align}
where $\obx_{t}$ is the network-averaged process, $\obx_{t}=(1/N)\sum_{l=1}^{N}\mathbf{x}_{l}(t)$.
\end{lemma}
\begin{proof}
Noting that $\left(\mathbf{1}_{N}\otimes I_{d}\right)^{\top}\left(L_{t}\otimes I_{d}\right)=\mathbf{0}$ (by the properties of the undirected Laplacian), we have by~\eqref{eq:vec_update},
\begin{align}
\label{lm:mean_conv_as2}
\obx_{t+1}=\obx_{t}-\alpha_{t}\left(\frac{1}{N}\sum_{n=1}^{N}\nabla U_{n}(\mathbf{x}_{n}(t))+\obzeta_{t}\right)+\gamma_{t}\omw_{t},
\end{align}
where
\begin{align}
\label{lm:mean_conv_as3}
\obzeta_{t}=\frac{1}{N}\sum_{n=1}^{N}\bzeta_{n}(t),~~\mbox{and}~~\omw_{t}=\frac{1}{N}\sum_{n=1}^{N}\mathbf{w}_{n}(t).
\end{align}
Denote by $\{\brx_{t}\}$ the process, $\brx_{t}=\mathbf{x}_{t}-\mathbf{1}_{N}\otimes\obx_{t}$, for all $t\geq 0$, and note that
\begin{align}
\label{special1}
\mathcal{P}_{Nd}\brx_{t}=\mathbf{0},~~\forall t,
\end{align}
since $\brx_{t}\in\mathcal{C}^{\perp}$, where recall $\mathcal{C}^{\perp}$ is the orthogonal complement of the consensus subspace (see~\eqref{eq:cons_subs}) and $\mathcal{P}_{Nd}=(1/N)\left(\mathbf{1}_{N}\otimes I_{d}\right)\left(\mathbf{1}_{N}\otimes I_{d}\right)^{\top}$.

By~\eqref{eq:vec_update} and~\eqref{lm:mean_conv_as2} we have
\begin{align}
\nonumber
\brx_{t+1} & = \left(I_{Nd}-\beta_{t}\left(L_{t}\otimes I_{d}\right)\right)\brx_{t} \\
\label{lm:mean_conv_as4} &
-\alpha_{t}\underbrace{\begin{bmatrix}\nabla U_{1}(\mathbf{x}_{1}(t))-\frac{1}{N}\sum_{l=1}^{N}\nabla U_{l}(\mathbf{x}_{l}(t)) \\  \nabla U_{2}(\mathbf{x}_{2}(t))-\frac{1}{N}\sum_{l=1}^{N}\nabla U_{l}(\mathbf{x}_{l}(t)) \\ \vdots \\ \vdots \\ \nabla U_{N}(\mathbf{x}_{N}(t))-\frac{1}{N}\sum_{l=1}^{N}\nabla U_{l}(\mathbf{x}_{l}(t))
\end{bmatrix}}_{\mathrm{\mbox{$T_{1}$}}} \\
&-\alpha_{t}\underbrace{\left(\bzeta_{t}-\mathbf{1}_{N}\otimes\obzeta_{t}\right)}_{\mathrm{\mbox{$T_{2}$}}}+\gamma_{t}\underbrace{\left(\mathbf{w}_{t}-\mathbf{1}_{N}\otimes\omw_{t}\right)}_{\mathrm{\mbox{$T_{3}$}}}
\end{align}
for all $t\geq 0$. (For convenience, we suppress the time index on the $T_i$ terms.) Now, consider the $n$-th component of the term $T_{1}$,
\begin{align}
\label{lm:mean_conv_as5}
T_{1}^{n}\doteq \nabla U_{n}(\mathbf{x}_{n}(t))-\frac{1}{N}\sum_{l=1}^{N}\nabla U_{l}(\mathbf{x}_{l}(t)),
\end{align}
and note that $T_{1}^{n}$ may be decomposed as
\begin{align}
\label{lm:mean_conv_as6}
T_{1}^{n} & =\left(\nabla U_{n}(\mathbf{x}_{n}(t))-\nabla U_{n}(\obx_{t})\right)\\
& +\left(\nabla U_{n}(\obx_{t})-\frac{1}{N}\sum_{l=1}^{N}\nabla U_{l}(\obx_{t})\right)\\
& +\left(\frac{1}{N}\sum_{l=1}^{N}\nabla U_{l}(\obx_{t})-\frac{1}{N}\sum_{l=1}^{N}\nabla U_{l}(\mathbf{x}_{l}(t))\right).
\end{align}
For the second term on the R.H.S. of~\eqref{lm:mean_conv_as6}, note that, by Assumption~\ref{ass:similarity}, there exists a constant $c_{1}>0$ such that
\begin{align}
\label{lm:mean_conv_as7}
\left\|\nabla U_{n}(\obx_{t})-\frac{1}{N}\sum_{l=1}^{N}\nabla U_{l}(\obx_{t})\right\|=\left\|\nabla U_{n}(\obx_{t})-\nabla U(\obx_{t})\right\|\leq c_{1}.
\end{align}
Finally, by the Lipschitz continuity of the gradients (see Assumption~\ref{ass:Lip}), we have, for a constant $c_{2}>0$ large enough,
\begin{align}
\label{lm:mean_conv_as8}
\left\|T_{1}^{n}\right\|\leq c_{1}+c_{2}\sum_{l=1}^{N}\left\|\mathbf{x}_{l}(t)-\obx_{t}\right\|.
\end{align}
Hence, there exist constants $c_{3}, c_{4}>0$ such that
\begin{align}
\label{lm:mean_conv_as9}
\left\|T_{1}\right\|\leq c_{3}+c_{4}\left\|\mathbf{x}_{t}-\mathbf{1}_{N}\otimes\obx_{t}\right\|=c_{3}+c_{4}\left\|\brx_{t}\right\|.
\end{align}
For the term $T_{2}$ in~\eqref{lm:mean_conv_as4}, consider $\delta\in (0, 1/2)$ arbitrarily small. Consider the process $\{R_{t}\}$, defined as $R_{t}=(t+1)^{-1/2-\delta}\left\|\bzeta_{t}-\mathbf{1}_{N}\obzeta_{t}\right\|$ for all $t$, and note that by Lemma~\ref{lm:zeta_decay} we have $R_{t}\rightarrow 0$ as $t\rightarrow\infty$ a.s. Since $t^{-1}\leq 2(t+1)^{-1}$ for all $t>0$, we have (see Assumption~\ref{ass:weights})
\begin{align}
\label{lm:mean_conv_as10}
\left\|\alpha_t T_{2}\right\|=\alpha_{t}(t+1)^{1/2+\delta}R_{t}\leq \frac{2c_\alpha}{(t+1)^{1/2-\delta}}R_{t},~~\mbox{for $t$ large.}
\end{align}
Similarly, note that,
\begin{align}
\label{lm:mean_conv_as11}
\left\|\gamma_{t}T_{3}\right\| & \leq \frac{2c_\gamma\|T_3\|}{(t+1)^{1/2}\sqrt{\log\log t}}\\
& \leq \frac{2c_\gamma}{(t+1)^{1/2-\delta}}\left\|T_{3}\right\|,~~\mbox{for $t$ large.}
\end{align}
Noting that $\|T_{3}\|$ has moments of all order (by the Gaussianity of the $\mathbf{w}_{t}$'s), by~\eqref{lm:mean_conv_as10}-\eqref{lm:mean_conv_as11} we conclude that there exist $\mathbb{R}_{+}$-valued $\{\mathcal{H}_{t+1}\}$-adapted processes $\{V_{t}\}$ and $\{J_{t}\}$ such that
\begin{align}
\label{lm:mean_conv_as12}
\left\|\alpha_{t} T_{2}\right\|+\left\|\gamma_{t}T_{3}\right\|\leq \frac{1}{(t+1)^{1/2-\delta}}V_{t}\left(1+J_{t}\right),~~\mbox{for $t$ large},
\end{align}
with $\{V_{t}\}$ being bounded a.s. and $\{J_{t}\}$ possessing moments of all orders.

Since $\brx_{t}\in\mathcal{C}^{\perp}$ for all $t\geq 0$, by Lemma~\ref{lm:conn} there exists a $\{\mathcal{H}_{t+1}\}$ adapted $\mathbb{R}_{+}$ valued process $\{r_{t}\}$ and a constant $c_{5}>0$ such that $0\leq r_{t}\leq 1$ a.s. and
\begin{align}
\label{lm:mean_conv_as13}
\left\|I_{Nd}-\beta_{t}\left(L_{t}\otimes I_{d}\right)\brx_{t}\right\|\leq \left(1-r_{t}\right)\left\|\brx_{t}\right\|
\end{align}
with
\begin{align}
\label{lm:mean_conv_as14}
\mathbb{E}\left[r_{t}~|~\mathcal{H}_{t}\right]\geq\frac{c_{5}}{(t+1)^{\tau_{\beta}}}~~\mbox{a.s.}
\end{align}
for all $t$ large enough.

Thus, by \eqref{lm:mean_conv_as4}, \eqref{lm:mean_conv_as9}, \eqref{lm:mean_conv_as12}, and \eqref{lm:mean_conv_as13} we obtain
\begin{align}
\label{lm:mean_conv_as15}
\left\|\brx_{t+1}\right\|\leq & \left(1-r_{t}\right)\left\|\brx_{t}\right\|\\
& +\alpha_{t}c_{3}+\alpha_{t}\left\|\brx_{t}\right\|+\frac{1}{(t+1)^{1/2-\delta}}V_{t}\left(1+J_{t}\right)
\end{align}
for $t$ large. Denote by $\{\br_{t}\}$ the process given by, $\br_{t}=r_{t}-\alpha_{t}$ for all $t$, and note that, since $\tau_{\beta}<1$, by~\eqref{lm:mean_conv_as14} there exists a constant $c_{6}>0$ such that
\begin{align}
\label{lm:mean_conv_as17}
\mathbb{E}\left[\br_{t}~|~\mathcal{H}_{t}\right]\geq\frac{c_{6}}{(t+1)^{\tau_{\beta}}}~~\mbox{a.s.}
\end{align}
for all $t$ large enough. Noting that $\alpha_{t}=c_\alpha t^{-1}\leq 2(t+1)^{-1/2+\delta}$ for all $t$ large, by~\eqref{lm:mean_conv_as15} we have
\begin{align}
\label{lm:mean_conv_as18}
\left\|\brx_{t+1}\right\|\leq\left(1-\br_{t}\right)\left\|\brx_{t}\right\|+\frac{c_{7}}{(t+1)^{1/2-\delta}}V_{t}\left(1+J_{t}\right)
\end{align}
for $t$ large and a constant $c_{7}>0$ sufficiently large. By~\eqref{lm:mean_conv_as17} and the above development, the recursion in~\eqref{lm:mean_conv_as18} clearly falls under the purview of Lemma~\ref{lm:mean-conv}, and we conclude that (by taking $\delta_{2}=1/2-\delta$ and $\delta_{1}=\tau_{\beta}$ in Lemma~\ref{lm:mean-conv}) for all
$\tau$ and $\varepsilon_{1}>0$ such that
\begin{align}
\label{lm:mean_conv_as19}
0\leq\tau<\frac{1}{2}-\delta-\tau_{\beta}-\frac{1}{2+\varepsilon_{1}},
\end{align}
we have $(t+1)^{\tau}\brx_{t}\rightarrow 0$ a.s. as $t\rightarrow\infty$. By taking $\varepsilon_{1}\rightarrow\infty$ (since $J_{t}$ has moments of all orders) and $\delta\rightarrow 0$, we conclude that $(t+1)^{\tau}\brx_{t}\rightarrow 0$ a.s. as $t\rightarrow\infty$ for all $\tau\in (0, 1/2-\tau_{\beta})$.
\end{proof}

\subsection{Classical Results: Recursive Algorithms for Global Optimization} \label{sec:GM-result}
We will now briefly review classical results on global optimization from~\cite{gelfand1991recursive} that will be used in the proof of our main result.

Consider the following stochastic recursion in $\mathbb{R}^{d}$:
\begin{align}
\label{eq:GM_alg}
\mathbf{z}_{t+1}=\mathbf{z}_{t}-a_{t}\left(\nabla U(\mathbf{z}_{t})+\bxi_{t}\right)+b_{t}\mathbf{w}_{t},~~~t\geq 0,
\end{align}
where $U:\mathbb{R}^{d}\to\mathbb{R}_{+}$, $\{\bxi_{t}\}$ is a sequence of $\R^d$-valued random variables, $\{\mathbf{w}_{t}\}$ is a sequence of independent $d$-dimensional Gaussian random variables with mean zero and covariance $I_{d}$, and
\begin{align}
\label{eq:GM_alg1}
a_{t}=\frac{A}{t},~~b^{2}_{t}=\frac{B}{t\log\log t},~~~\mbox{for $t$ large},
\end{align}
where $A,B>0$ are constants.

Consider the following assumptions on $U(\cdot)$, the gradient field $\nabla U(\cdot)$, and noise $\bxi_t$:
\begin{assumption}
\label{ass:GM_1} $U:\R^d\to\R$ is a $C^2$ function such that
\begin{enumerate}
\item[(i)] $\min U(x) = 0$,
\item[(ii)] $U(x)\to\infty$ and $|\nabla U(x)|\to\infty$ as $|x|\to\infty$,
\item[(iii)] $\inf(|\nabla U(x)|^2 - \Delta U(x) ) > -\infty$.
\end{enumerate}
\end{assumption}
We note that Within the context of PAC learning, the assumption (i) above corresponds to the ``realizability'' assumption, i.e., there exists a true (but unknown) hypothesis that accurately represents that data.
\begin{assumption}
\label{ass:GM_2} For $\e>0$ let
$$
d\pi^\e(x) = \frac{1}{Z^\e}\exp\left(-\frac{2U(x)}{\e^2} \right)\dx,
$$
$$
Z^\e= \int\exp\left(-\frac{2U(x)}{\e^2} \right)\dx.
$$
$U$ is such that $\pi^\e$ has a weak limit $\pi$ as $\e\to 0$.
\end{assumption}
We note that $\pi$ is constructed so as to place mass 1 on the set of global minima of $U$. A discussion of simple conditions ensuring the existence of such a $\pi$ can be found in \cite{hwang1980laplace}.
\begin{assumption}
\label{ass:GM_3} $\liminf_{|x|\to\infty}\langle \frac{\nabla U(x)}{|\nabla U(x)|}, \frac{x}{|x|} \rangle \geq C(d)$,
$C(d) = \left( \frac{4d-4}{4d-3} \right)^{\frac{1}{2}}$.
\end{assumption}
\begin{assumption}
\label{ass:GM_5} $\liminf_{|x|\to\infty} \frac{|\nabla U(x)|}{|x|} > 0$
\end{assumption}
\begin{assumption}
\label{ass:GM_6}
$\limsup_{|x|\to\infty} \frac{|\nabla U(x)|}{|x|} < \infty$
\end{assumption}

Let $\{\calG_t\}$ be the natural filtration generated by \eqref{eq:GM_alg}; that is, $\calG_t$, $t\geq 0$ is given by
$$
\calG_t = \sigma(\{\vx_0, \bxi_1,\ldots,\bxi_{t-1},\vw_1,\ldots,\vw_{t-1} \})
$$
\begin{assumption}
\label{ass:GM_4}
There exists $C_1>0$ such that
$$
\E\left( \left\|\bxi_t\right\|^2\vert \calG_t \right) \leq C_1 a_t^{\gamma_1}, \quad \left\|E(\xi_t\vert\calG_t)\right\| \leq C_1 a_t^{\gamma_2}~~ a.s.
$$
with $\gamma_1 > -1$ and $\gamma_2>0$.
\end{assumption}
Note that, in contrast to Assumption~\ref{ass:grad_noise}, Assumption~\ref{ass:GM_4} assumes the conditional mean may be non-zero (but decaying).

Finally, let $C_{0}$ be the constant as defined after (2.3) in~\cite{gelfand1991recursive}.

The following result on the convergence of the stochastic process \eqref{eq:GM_alg} was obtained in~\cite{gelfand1991recursive}.
\begin{theorem}[Theorem 4 in~\cite{gelfand1991recursive}]
\label{th:GM} Let $\{\mathbf{z}_{t}\}$ be given by~\eqref{eq:GM_alg}. Suppose that Assumptions \ref{ass:GM_1}--\ref{ass:GM_4} hold %with $\gamma_{1}>-1$ and $\gamma_{2}>0$,
%and suppose that the gradient noise sequence $\{\zeta_t\}$ satisfies Assumption \ref{ass:grad_noise}.
and assume $A$ and $B$ in \eqref{eq:GM_alg1} satisfy $B/A>C_{0}$. Then, for any bounded continuous function $f:\mathbb{R}^{d}\to\mathbb{R}$, we have that
\begin{align}
\label{th:GM1}
\lim_{t\rightarrow\infty}\mathbb{E}_{0,\mathbf{z}_{0}}\left[f(\mathbf{z}_{t})\right]=\pi(f).
\end{align}
%uniformly for $\mathbf{z}_{0}$ in a compact set.
\end{theorem}

\section{Main Results}
\label{sec:main_res}
We will now prove the main result of the paper. We shall proceed as follows. We will first study the behavior of the $\R^d$-valued networked averaged process
\begin{equation} \label{eq:avg}
\obx_{t}=\frac{1}{N}\sum_{n=1}^{N}\mathbf{x}_{n}(t), \quad\quad t\geq 1.
\end{equation}
Using Theorem~\ref{th:GM}, we will show that $\bar \vx_t$ converges to
the set of global minima of $U(\cdot)$ (see Lemma~\ref{lm:avg_conv}). After proving Lemma~\ref{lm:avg_conv} we will present Theorem~\ref{th:main_result}, which is the main result of the paper. Theorem~\ref{th:main_result} follows as a straightforward consequence of Lemmas~\ref{lm:mean_conv_as} and \ref{lm:avg_conv}.
%Theorem \ref{th:main_result} states that each individual $\vx_n(t)$, $n=1,\ldots,N$ converges to the set of global minima of $U(\cdot)$. Theorem \ref{th:main_result} follows readily from Lemmas \ref{lm:mean_conv_as} and \ref{lm:avg_conv}.

Note that, taking the average on both sides of~\eqref{eq:update} we obtain,
\begin{align}
\label{eq:avg_update}
\obx_{t+1}=\obx_{t}-\alpha_{t}\left(\nabla U(\obx_{t})+\obzeta_{t}+\oR_{t}\right)+\gamma_{t}\omw_{t},
\end{align}
where
\begin{align}\oR_{t}=\frac{1}{N}\sum_{n=1}^{N}\left(\nabla U_{n}(\mathbf{x}_{n}(t))-\nabla U_{n}(\obx_{t})\right),
\end{align}
and $\obzeta_t$ and $\omw_t$ are given in \eqref{lm:mean_conv_as3}.

The following lemma shows that the networked-averaged process $\{\overline \vx_t\}$ converges to the set of global minima of $U(\cdot)$.
\begin{lemma}
\label{lm:avg_conv}
Let $\{\vx_t\}$ satisfy the recursion \eqref{eq:vec_update} and let $\{\overline \vx_t\}$ be given by \eqref{eq:avg}, with initial condition $\bold{\bar x_0}\in \R^d$. Let Assumptions \ref{ass:conn}--\ref{ass:weights} hold and Assume $U(\cdot)$ satisfies Assumptions \ref{ass:Lip}--\ref{ass:similarity} and \ref{ass:GM_1}--\ref{ass:GM_6}.
%Let Assumptions~\ref{ass:Lip}--?? and ??--\ref{ass:GM_6} hold.
%\ref{ass:weights} and Assumptions~\ref{ass:GM_1}-\ref{ass:GM_6} hold.
Further, suppose that $c_{\alpha}$ and $c_{\gamma}$ in Assumption~\ref{ass:weights} satisfy, $c_{\gamma}^{2}/c_{\alpha}>C_{0}$, where $C_0$ is defined after Assumption~\ref{ass:GM_4}. Then, for any bounded continuous function $f:\mathbb{R}^{d}\to\mathbb{R}$, we have that
\begin{align}
\label{lm:avg_conv1}
\lim_{t\rightarrow\infty}\mathbb{E}_{0,x_0}\left[f(\obx_{t})\right]=\pi(f).
\end{align}
%uniformly for $\bar x_0$ in a compact set.
\end{lemma}
\begin{proof}
The result will be proven by showing that the $\{\obx_{t}\}$ in \eqref{eq:avg_update} falls under the purview of Theorem~\ref{th:GM}, and, in particular, that Assumption~\ref{ass:GM_4} is satisfied. To do this, the key technical issue lies in handling the process $\{\oR_{t}\}$. Specifically, we must restate the a.s. convergence obtained in Lemma~\ref{lm:mean_conv_as} in terms of conditional expectations as required by Assumption~\ref{ass:GM_4}.

Fix $\tau\in (0, \frac{1}{2}-\tau_{\beta})$ and let $\delta>0$ be arbitrary. Since, by Lemma~\ref{lm:mean_conv_as}, $t^{\tau}\|\mathbf{x}_{n}(t)-\obx_{t}\|\rightarrow 0$ as $t\rightarrow 0$ a.s. for all $n$, by Egorov's theorem there exists a constant $R_{\delta}>0$ such that
\begin{align}
\label{lm:avg_conv2}
\mathbb{P}\left(\sup_{t\geq 0} t^{\tau}\left\|\mathbf{x}_{n}(t)-\obx_{t}\right\|\leq R_{\delta}\right)>1-\delta,~~\forall n.
\end{align}
Note that, by Assumption~\ref{ass:Lip},
\begin{align}
\label{lm:avg_conv3}
\mathbb{P}\left(\sup_{t\geq 0} t^{\tau}\left\|\mathbf{x}_{n}(t)-\obx_{t}\right\|\leq R_{\delta}\right)\leq \mathbb{P}\left(\sup_{t\geq 0} t^{\tau}\left\|\oR_{t}\right\|\leq LR_{\delta}\right).
\end{align}
Now consider the $\mathcal{H}_{t}$-adapted process $\{\oR^{\delta}_{t}\}$, given by
\begin{align}
\label{lm:avg_conv4}
\oR^{\delta}_{t}=\left\{\begin{array}{ll}
                                \oR_{t} & \mbox{if $t^{\tau}\|\oR_{t}\|\leq LR_{\delta}$}\\
                                \frac{LR_{\delta}}{t^{\tau}} & \mbox{if $t^{\tau}\|\oR_{t}\|> LR_{\delta}$,}\\
                                \end{array}
                                \right.
\end{align}
for all $t\geq 0$. Note that, by construction,
\begin{align}
\label{lm:avg_conv5}
\mathbb{P}\left(\sup_{t\geq 0} t^{\tau}\left\|\oR_{t}^{\delta}\right\|\leq LR_{\delta}\right)=1.
\end{align}
Now consider the stochastic process $\{\obx^{\delta}_{t}\}$ evolving as
\begin{align}
\label{lm:avg_conv6}
\obx_{t+1}^{\delta}=\obx_{t}^{\delta}-\alpha_{t}\left(\nabla U(\obx^{\delta}_{t})+\obzeta_{t}+\oR^{\delta}_{t}\right)+\gamma_{t}\omw_{t}
\end{align}
with initial condition $\obx_{0}^{\delta}=\frac{1}{N}\sum_{n=1}^{N}\mathbf{x}_{n}(0)$. It is readily seen that $\{\obx_{t}^{\delta}\}$ is $\mathcal{H}_{t}$-adapted and the processes $\{\obx_{t}^{\delta}\}$ and $\{\obx_{t}\}$ agree, for all $t\geq 0$, on the event $\left\{\sup_{t\geq 0}t^{\tau}\left\|\oR_{t}\right\|\leq LR_{\delta}\right\}$, so that
\begin{align}
\label{lm:avg_conv7}
\mathbb{P}\left(\sup_{t\geq 0}\left\|\obx_{t}^{\delta}-\obx_{t}\right\|>0\right)\leq \delta.
\end{align}
Let $\obxi_{t}=\obzeta_{t}+\oR^{\delta}_{t}$ for all $t\geq 0$, and denote by $\mathcal{F}^{\delta}_{t}$ the $\sigma$-algebra
\begin{align}
\label{lm:avg_conv8}
\mathcal{F}^{\delta}_{t}=\sigma\left(\obx_{0}^{\delta}, \obxi_{0},\cdots,\obxi_{t-1},\omw_{0},\cdots,\omw_{t-1}\right).
\end{align}
Note that $\mathcal{F}_{t}^\delta\subset\mathcal{H}_{t}$ for all $t\geq 0$.

By Assumption~\ref{ass:grad_noise} and~\eqref{lm:avg_conv5} we have that, almost surely,
\begin{align}
\label{lm:avg_conv9}
\left\|\mathbb{E}\left[\left(\obzeta_{t}+\oR^{\delta}_{t}\right)~|~\mathcal{F}_{t}^\delta\right]\right\|\leq\left\|\mathbb{E}\left[\oR^{\delta}_{t}~|~\mathcal{F}_{t}^\delta\right]\right\|\leq \frac{LR_{\delta}}{t^{\tau}}
\end{align}
and, by the parallelogram law,
\begin{align}
\label{lm:avg_conv10}
\mathbb{E}\left[\left\|\obzeta_{t}+\oR^{\delta}_{t}\right\|^{2}~|~\mathcal{F}_{t}^\delta\right] \leq & 2\mathbb{E}\left[\left\|\obzeta_{t}\right\|^{2}~|~\mathcal{F}_{t}^\delta \right] \\
&~+ 2\mathbb{E}\left[\left\|\oR^{\delta}_{t}\right\|^{2}~|~\mathcal{F}_{t}^\delta\right]\\ \leq & 2C_{1}+\frac{2L^{2}R^{2}_{\delta}}{t^{2\tau}}.
\end{align}
It is now straightforward to see that the process $\{\obx^{\delta}_{t}\}$ falls under the purview of Theorem~\ref{th:GM}. In particular, letting by taking $\gamma_{1}=0$ and $\gamma_{2}=\tau >0$ and letting $\calG_t =\calF_t^\delta$ we see that Assumption~\ref{ass:GM_4} is satisfied.
We conclude that
\begin{align}
\label{lm:avg_conv11}
\lim_{t\rightarrow\infty}\mathbb{E}\left[f(\obx_{t}^{\delta})\right]=\pi(f).
\end{align}
Note that, by~\eqref{lm:avg_conv7},
\begin{align}
\label{lm:avg_conv12}
\left|\mathbb{E}\left[f(\obx_{t})\right]-\pi(f)\right| \leq & \mathbb{E}\left[\left|f(\obx_{t})-f(\obx^{\delta}_{t})\right|\right]\\
& +\left|\mathbb{E}\left[f(\obx_{t}^{\delta})\right]-\pi(f)\right| \\
 \leq & 2\left\|f\right\|_{\infty}\delta+\left|\mathbb{E}\left[f(\obx_{t}^{\delta})\right]-\pi(f)\right|.
\end{align}
Hence, by~\eqref{lm:avg_conv11}, we have
\begin{align}
\label{lm:avg_conv13}
\limsup_{t\rightarrow\infty}\left|\mathbb{E}\left[f(\obx_{t})\right]-\pi(f)\right|\leq 2\left\|f\right\|_{\infty}\delta.
\end{align}
Since $\delta>0$ is arbitrary, we conclude that
\begin{align}
\label{lm:avg_conv14}
\lim_{t\rightarrow\infty}\left|\mathbb{E}\left[f(\obx_{t})\right]-\pi(f)\right|=0.
\end{align}
\end{proof}

We now state the main result concerning convergence of the agent estimates $\{\mathbf{x}_{n}(t)\}$ to the set of global minima of $U(\cdot)$.
\begin{theorem}
\label{th:main_result}
Let $\{\vx_t\}$ satisfy the recursion \eqref{eq:vec_update} with initial condition $\bold{x_0}$. Let Assumptions \ref{ass:conn}--\ref{ass:weights} hold and assume $U(\cdot)$ satisfies Assumptions \ref{ass:Lip}--\ref{ass:similarity} and \ref{ass:GM_1}--\ref{ass:GM_6}.
%Let Assumptions~\ref{ass:Lip}--?? and ??--\ref{ass:GM_6} hold.
%\ref{ass:weights} and Assumptions~\ref{ass:GM_1}-\ref{ass:GM_6} hold.
Further, suppose that $c_{\alpha}$ and $c_{\gamma}$ in Assumption~\ref{ass:weights} satisfy $c_{\gamma}^{2}/c_{\alpha}>C_{0}$, where $C_0$ is defined after Assumption~\ref{ass:GM_4}.
%Let Assumptions \ref{ass:Lip}--\ref{ass:GM_6} hold.
%Further, suppose that $c_{\alpha}$ and $c_{\gamma}$ in Assumption~\ref{ass:weights} satisfy, $c_{\gamma}^{2}/c_{\alpha}>C_{0}$.
Then, for any bounded continuous function $f:\mathbb{R}^{d}\to\mathbb{R}$ and for all $n=1,\ldots,N$, we have that
\begin{align}
\label{th:main_result1}
\lim_{t\rightarrow\infty}\mathbb{E}_{0,x_0}\left[f(\mathbf{x}_{n}(t))\right]=\pi(f).
\end{align}
%uniformly for $x_0$ in a compact set.
\end{theorem}
In the above theorem, we recall that we use the conventions \eqref{eq:pi-f-def}-- \eqref{eq:def-cond-E}, and that $\pi$ is a probability measure placing mass 1 on the set of global minima of $U(\cdot)$ as constructed in Assumption~\ref{ass:GM_2}. The proof of the theorem follows below.
\begin{proof} The proof follows immediately from Lemma~\ref{lm:mean_conv_as} and Lemma~\ref{lm:avg_conv}. In particular, by Lemma~\ref{lm:mean_conv_as}, we have that $\|\mathbf{x}_{n}(t)-\obx_{t}\|\rightarrow 0$ as $t\rightarrow\infty$ a.s. (by taking $\tau=0$ in Lemma~\ref{lm:mean_conv_as}) and the desired assertion follows by noting that the above a.s. convergence implies that the processes $\{\mathbf{x}_{n}(t)\}$ and $\{\obx_{t}\}$ have the same weak limits.
\end{proof}

\section{Conclusions}
The paper proves, for a distributed consensus + innovations algorithm, convergence to the set of global optima in a distributed nonconvex optimization problems. Each agent only has information about the gradient of its personal objective function and the current state of neighboring agents. Convergence to a global minimum is achieved by means of decaying (annealing) noise injected into each agent's update step. The paper proved convergence (in probability) of the algorithm to the set of global minima of the sum objective.
%Future work may consider relaxing some of the assumptions used in the paper.
%For example, Assumption \ref{ass:similarity}

%\bibliographystyle{siamplain}
\bibliographystyle{IEEEtran}
\bibliography{dist_glob_opt}

\end{document}